\documentclass[12pt, a4paper]{article}

\usepackage{cmap}
\usepackage{mathtext}
\usepackage[T2A]{fontenc}
\usepackage[utf8]{inputenc}

\usepackage[all]{xy}

\usepackage{leftidx}

\usepackage{hyperref}

\usepackage{amsmath,amssymb,amscd,mathrsfs,bm}

\usepackage{amsthm} 

\newcommand{\entropy}{\mathop{\mathsf{h}}\nolimits}
\newcommand{\ext}{\mathop{\mathrm{Ext}}\nolimits}

\newcommand{\gd}{\mathop{\mathrm{gl. dim}}\nolimits}
\newcommand{\gkd}{\mathop{\mbox{GK-dim}}\nolimits}

\newcommand{\Hom}{\mathop{\mathsf{Hom}}\nolimits}
\newcommand{\hpol}{\mathsf{h}^\mathrm{pol}}  
\newcommand{\inte}{\ensuremath{\mathbb Z}}
\newcommand{\kkk}{\ensuremath{\Bbbk}}

\newcommand{\Mod}{\mathop{\mathsf{Mod}}\nolimits}
\newcommand{\Ob}{\mathop{\mathsf{Ob}}\nolimits}
\newcommand{\PP}{\ensuremath{\mathbb P}}
\newcommand{\Qgr}{\mathop{\mathsf{Qgr}}\nolimits}
\newcommand{\qgr}{\mathop{\mathsf{qgr}}\nolimits}
\newcommand{\rad}{\mathop{\mathsf{rad}}\nolimits}
\newcommand{\real}{\ensuremath{\mathbb R}}
\newcommand{\RHom}{\mathop{\mathsf{RHom}}\nolimits}
\newcommand{\rk}{\mathop{\mathrm{rk}}\nolimits}
\newcommand{\tor}{\mathop{\mathrm{Tor}}\nolimits}

\newtheorem{theorem}{\quad Theorem}[section]
\newtheorem{definition}[theorem]{\quad Definition}
\newtheorem{corollary}[theorem]{\quad Corollary}
\newtheorem{proposition}[theorem]{\quad Proposition}
\newtheorem{lemma}[theorem]{\quad Lemma}

\newtheorem{conj}[theorem]{\quad Conjecture}

\theoremstyle{definition}
\newtheorem{example}[theorem]{\quad Example}

\begin{document}

\title{\textsf{Growth in noncommutative algebras and entropy in derived categories
}}
\author{Dmitri Piontkovski \thanks{HSE University, Myasnitskaya ul. 20, Moscow 101000, Russia, \sf{dpiontkovski@hse.ru}}}

\date{}


\maketitle

\begin{abstract}
A noncommutative projective variety is defined, following Artin and Zhang, by a graded coherent algebra $A$. The category of coherent sheaves is then the quotient $\mathsf{qgr}(A)$ of the category of finitely presented graded modules by the subcategory of torsion modules.  We consider the categorical and polynomial entropies of the Serre twist, that is, of the degree-shift functor on the bounded derived category $\mathbf{D}^b(\mathsf{qgr}(A))$. These two types of entropy can be viewed as analogues of  dimension for the noncommutative variety.

We relate these invariants to the growth of the algebra $A$.  For algebras of finite global dimension, the categorical and polynomial entropies of the Serre twist  are bounded above by the growth entropy and the Gelfand--Kirillov dimension of the algebra. Moreover, equality holds for regular algebras and for coordinate rings of smooth projective varieties.  In contrast, the  polynomial entropy vanishes for monomial algebras of polynomial growth, so in this case the inequality is strict. 
\end{abstract}

\section{Introduction}

Let $A$ be a finitely generated  graded coherent 
associative algebra with 1 over a field $k$.
Then 
it is natural to regard  $A$ as a coordinate ring of some  noncommutative projective variety~\cite{verevkin1992noncommutative}, \cite{Artin1994NoncommutativeSchemes}.
What is the dimension of such a  variety?

In the classical case of a commutative algebra $A$ and the associated projective variety $V$, one can calculate the dimension of the variety via  the Gelfand--Kirillov dimension of $A$, that is, $\dim V  = \gkd A -1$ for
$$\gkd A = \varlimsup_{n\to\infty} \frac{ \log  (\dim A_0+\dots+ \dim A_n)}{\log n} = 
\varlimsup_{n\to\infty} \frac{ \log  \dim A_n}{\log n} +1
.$$
However, in typical noncommutative examples (such as the case of free associative algebra $A$ with at least two generators) the growth of $A$ is exponential, so that the Gelfand--Kirillov dimension is infinite. 
 
In the case of exponential growth, the asymptotic behavior of the graded component dimensions  is measured by the growth entropy of $A$. We define it as the 
logarithm\footnote{Note that in the original definition~\cite{Newman2000TheAlgebras} the term {\em entropy} of graded algebra  is used for the growth exponent $H(A) = \varlimsup_{n\to \infty} \sqrt[n]{\dim A_n} $. Here  
we prefer the logarithmic form since 
it makes comparison with entropies of other kinds more natural, see~\cite{lu2023entropy} and  the results below.}
 of the growth exponent: 
\begin{equation*}
\mathsf{h} (A) = \log \varlimsup_{n\to \infty} \sqrt[n]{\dim A_n} = \varlimsup_{n\to \infty} \frac{\log \dim A_n}{n}.
\end{equation*}

In many cases the invariants $\gkd A$ 
and $\mathsf{h} (A)$ play the role of a dimension for noncommutative varieties. 
For example, for all known Noetherian ``noncommutative projective spaces'' $\PP^{d-1}$ defined by Artin--Schelter regular algebras $A$  of global dimension $d$, the Gelfand--Kirillov dimension  $\gkd A $ is equal to $d$~\cite[Sec.~3.3]{rogalski2024artin}. 
To look over examples of exponential growth, note that there is an infinite sequence $A^n$ 
of graded algebras corresponding to pairwise non-isomorphic versions of the noncommutative projective line~\cite{piontkovski2008coherent}. All algebras $A^n$ have the same global dimension (two) and the same Gelfand--Kirillov dimension (infinity), however,  one can distinguish these algebras using entropy, see Example~\ref{example:NCline}.

On the other hand, the dimension of a projective variety can be recovered in terms of the derived category of coherent sheaves in various ways (e.~g., see~\cite{elagin2021three} for   three such approaches). Can one define the dimension of a noncommutative variety in terms of the derived category of noncommutative coherent sheaves?

Suppose that the algebra $A$ is graded coherent. 
If $\mathsf{gr}(A)$ and $\mathsf{tors}( A)$ denote the categories of $\mathbb{Z}$-graded  finitely presented right modules and finite-dimensional such modules, one can consider the quotient category $\mathsf{qgr}(A):= \mathsf{gr}(A)/ \mathsf{tors} (A)$ 
as the category of coherent sheaves on the noncommutative variety. 
Following the approach of M.~Artin and Zhang \cite{Artin1994NoncommutativeSchemes} (based partially on the work by Verevkin~\cite{verevkin1992noncommutative} and earlier ideas of Manin), one can associate to the noncommutative variety the triple $(\qgr A, \mathcal{O}, s)$, where $\qgr A $ is  the abelian category of ``coherent sheaves'',  $\mathcal{O}$ is the ``structure sheaf''
 (the image of the right module $A_A$ in $\qgr A$), and  $s$ is the ``polarization'' (the autoequivalence of $\qgr A$ induced by the shift of grading 
$M \mapsto M[1]$ of graded modules). Under mild assumptions (which holds for many Noetherian and coherent regular algebras), one can recover the algebra $A$ from the triple~\cite[Sec.~4]{Artin1994NoncommutativeSchemes}, \cite{Polishchuk2005NoncommutativeAlgebras}. 
 However, generally the category $\qgr A$ and the whole triple contain less information than the algebra $A$.  

How much information about a noncommutative variety is retained by the bounded derived category of coherent sheaves $\mathbf{D}^b(\mathsf{qgr}(A))$?
Let 
$S$ denotes the Serre twist
on $\mathbf{D}^b (\mathsf{qgr}(A))$, that is, the endofunctor induced by the degree shift $s$.\footnote{The Serre twist functor should not be confused with either the Serre functor in the sense of Bondal and Kapranov or the translation functor of the triangulated category.}  
In important cases, the triangulated category $\mathbf{D}^b (\mathsf{qgr}(A))$ admits a classical generator. Then the complexity of 
the endofunctor $S$ is measured by its categorical entropy $\mathsf{h}_t (S)$~\cite{Dimitrov2014DynamicalCategories} 
and polynomial entropy $\hpol_t(S) $~\cite{fan2021categorical}.
Generally, these entropies are functions of a real variable $t$, but we see that in important cases they are constant  
equal to the growth entropy and the Gelfand--Kirillov dimension of the algebra.
Thus, 
these two types of entropy can be regarded as dimension-like invariants of the variety insofar as the dimension is reflected by the derived category.

In the classical projective geometry setting, the connection between asymptotic invariants of a commutative algebra and the corresponding derived invariant of a smooth  projective variety follows from the results of~\cite{fan2021categorical}: 

\begin{proposition}[Proposition~\ref{prop: smooth_proj}]
\label{prop: intro: smooth_proj}
Let $X$ be a smooth projective variety with the coordinate ring $A$ (so that $\dim X = \gkd A - 1$). Then 
$
\mathsf{h}_t (S) = 0 =  \mathsf{h} (A) 
$
and
$$
\hpol_t(S) = \gkd A-1. 
$$
\end{proposition}

For a widely known class of noncommutative smooth algebras (that is, the class of regular algebras) we obtain an analogous result. It holds for both categorical and polynomial entropy.

\begin{theorem}[Corollary~\ref{cor:reg}]
\label{th: intro: AS-reg}
Let $A$ be  connected regular algebra. If $A$ is graded coherent, then 
the category $\mathbf{D}^b(\mathsf{qgr}(A))$ admits a split generator and 
$$
\mathsf{h}_t (S) =  \mathsf{h} (A). 
$$
If $A$ has polynomial growth (that is, $A$ is Artin--Schelter regular), then
$$
\hpol_t(S) = \gkd A-1. 
$$
\end{theorem}

If the ground field $k$ is perfect, the conclusion of Theorem~\ref{th: intro: AS-reg} holds for non-connected graded coherent  regular algebras as well, see Corollary~\ref{cor:reg}.  

For example, let $Q$ be a finite acyclic quiver of infinite representation type. Then its preprojective algebra $\Pi_kQ $ is regular and graded coherent  with respect to certain locally finite grading.  Thus, the algebra  $\Pi_kQ $ satisfies  the conclusion of Theorem~\ref{th: intro: AS-reg}, see Corollary~\ref{cor:prepr}.

The proofs of the above equalities are based on the smoothness of the corresponding commutative and noncommutative varieties.
Nevertheless, although a monomial algebra of exponential growth does not in any sense define a smooth variety, its entropy also satisfies an analogous inequality.


\begin{theorem}[\cite{lu2023entropy}]
Suppose that $A$ is a quotient algebra of a path algebra of a finite quiver by a finite set of monomial relations. Then $A$ is coherent, the category  $\mathbf{D}^b (\mathsf{qgr}(A))$ admits a classical generator, and the categorical entropy of the Serre twist equals to the growth entropy of the algebra:
$$
\mathsf{h}_t (S) = \mathsf{h} (A) . 
$$
\end{theorem} 

For algebras of polynomial growth the growth entropy is zero, so that in the case of monomial algebras of polynomial growth the above theorem gives the trivial equality $0=0$. One may expect that in this case  the polynomial categorical entropy  is equal to the Gelfand--Kirillov dimension of the algebra. Surprisingly, we have the following result. 

\begin{theorem}
Suppose that $A$ is a quotient algebra of a path algebra of a finite quiver by a finite set of monomial relations. If $A$ has polynomial growth, then the polynomial categorical entropy is zero, 
$$
\hpol_t(S)  = 0. 
$$
\end{theorem}

After the above considerations, 
we state the following two conjectures.

\begin{conj}
\label{conj: exp}
\label{conj: pol}
Suppose that the category $\mathbf{D}^b (\mathsf{qgr} A)$ for a graded coherent algebra $A$ admits a classical generator. 

(a) The entropy of the Serre twist functor on this category satisfies the inequality
$$
\mathsf{h}_t (S) \le  \entropy (A) . 
$$

(b) If the algebra $A$ has subexponential growth, then the polynomial entropy of this functor satisfies the inequality 
$$
\hpol_t (S) \le   \gkd A  -1. 
$$
\end{conj}

We prove these conjectures for algebras of finite global dimension.
It follows from the results of~\cite[Section~4]{Bondal2003GeneratorsGeometry}
that the category $\mathbf{D}^b (\mathsf{qgr} A)$ admits in this case a classical generator. 

\begin{theorem}[Corollary~\ref{cor:main_finite_dim}]
Suppose that a finitely generated connected graded $\kkk$-algebra $A$ is right graded coherent and has finite global dimension. Then
the entropy of the Serre twist functor on this category satisfies the inequality
$$
\mathsf{h}_t (S) \le  \entropy (A) . 
$$
Moreover, if the algebra has subexponential growth, then the polynomial entropy of this functor satisfies the inequality 
$$
\hpol_t (S) \le   \gkd A  -1. 
$$
So,  both Conjectures~\ref{conj: exp} (a) and (b)
hold for $A$.
\end{theorem}

We still have no examples of graded coherent algebras which satisfy the strict inequality 
$
\mathsf{h}_t (S) <  \entropy (A) . 
$
On the other hand, it seems important to describe new classes of algebras for which the equalities $
\mathsf{h}_t (S) =  \mathsf{h} (A)
$ and $
\hpol_t(S) = \gkd A-1 
$
hold. 


The paper is organized as follows. In Section~\ref{sec:back}, we introduce the rest of notations and give a briefly survey  of graded algebras and of entropy in triangulated categories. Some other necessary definitions are recalled later in their places.  In Section~\ref{sec: smooth}, we prove the equalities in the ``smooth'' cases. For this purpose,  we use here the calculation of complexity via so-called $\ext$-distance~\cite{Dimitrov2014DynamicalCategories}  and the properties of the bounded derived category of finite-dimensional algebras~\cite{elagin2020smoothness}. 
In the next Section~\ref{sec: inec}, we use the results 
of~\cite{Bondal2003GeneratorsGeometry} to prove the inequalities for entropy in the case of algebras of  
finite global dimension. In final Section~\ref{sec: mon_pol}, we show that the polynomial entropy for monomial algebras of polynomial growth is zero. Our first point here is the connection of the complexity with the rank of the noncommutative sheaves~\cite{lu2023entropy}.

\subsection{Acknowledgement}

The author is grateful to Alexander Efimov, Lu Li, and Denis Lyskov for fruitful discussions.  
The paper is partially supported  by RSF grant No. 24-21-00341.

\section{Background}

\label{sec:back}

\subsection{Graded algebras and noncommutative varieties}

We consider graded finitely generated $k$-algebras of the form $A = A_0 \oplus A_1\oplus A_2\oplus \dots$ 
Such an algebra $A$ is called connected if $A_0 = k$. A connected algebra is {\em standard} if it is generated by $A_0 $ and $A_1$.   Recall that  a (graded) algebra $A$ is called (graded) right coherent if all finitely generated right-sided ideals in $A$ are finitely presented or, equivalently, all finitely presented right (graded) $A$-modules form an abelian category. We refer to right graded coherent algebras simply as graded coherent.  
All modules are right-sided. 

In addition to the algebraic entropy and the Gelfand-Kirillov dimension, we will sometimes consider little bit more general growth measures suitable for bi-graded algebras. 
Suppose that the algebra $A =\bigoplus_{n\ge 0}\bigoplus_{m\in \inte} A_{n,m}$ is bi-graded. 
Consider the Hilbert series of the variable $e^{-t}$ for the $n$-th graded component $A_n = \bigoplus_{m\in \inte} A_{n,m}$ (which is a graded module over the usual graded algebra $A_{0}$):
$$
 H_{A_n}(e^{-t}) = \sum_{m\in \inte} e^{-mt} \dim A_{n,m}.
$$  
For usual graded algebras, we assume that the second grading is trivial (i.e., $A_{n,m} =0 $ for $m\ne 0$), so that in this case
$ H_{A_n}(e^{-t}) = \dim A_{n,0}=\dim A_n$. 

  Define the algebraic entropy and the polynomial entropy  of the bigraded algebra as real functions $\mathsf \mathsf{h}_t(A), \hpol_t(A) : \mathbb{R} \rightarrow [-\infty,+\infty]$ of $t$:  
\begin{equation}
\mathsf{h}_t (A) =  \varlimsup_{n\to \infty} \frac{\log H_{A_n}(e^{-t}) }{n}
\end{equation}
and 
\begin{equation}
\label{eq:hpol_alg_pol_growth}
\hpol_t(A) = \limsup_{n \to \infty} \frac{\log H_{A_n}(e^{-t})  - n \mathsf{h}_t(A)}{\log n}. 
\end{equation}
If  the entropy $\mathsf{h}_t (A)$ is the zero function, the formula for the polynomial entropy gives 
$$
\hpol_t(A) = \limsup_{n \to \infty} \frac{\log H_{A_n}(e^{-t})}{\log n}. 
$$
For the same algebra with the total grading $A_* = \bigoplus_{n\ge 0} A_n$, its entropy is recovered from the value of teh function 
$\mathsf{h}_t (A) $ at $t=0$ as 
$$
\mathsf{h} (A_*)  = \mathsf{h}_0 (A) .
$$ Moreover, if $A$  is finitely generated and 
has subexponential growth (so that $\mathsf{h} (A_*)  = \mathsf{h}_0 (A) =0$), then we deduce from~(\ref{eq:hpol_alg_pol_growth}) that
  $$
\gkd A =   \hpol_0(A) +1. 
$$
Later we will discuss these entropies for bigraded orbit algebras of endofunctors. 

Given a graded coherent algebra $A$, let   $\mathsf{Gr}(A)$ and $\mathsf{gr}(A)$ denote the abelian 
categories of  all graded right $A$-modules and finitely presented such modules, respectively.
Denote by $\mathsf{Tors}( A)$ and $\mathsf{tors}( A)$ the full subcategories
of $\mathsf{Gr}(A)$ and $\mathsf{gr}(A)$ consisting 
of all (respectively, finite) torsion modules.
Then the quotient category $\mathsf{Qgr}(A):= \mathsf{Gr}(A)/ \mathsf{Tors} (A)$ plays the role of the of the category of quasicoherent sheaves on the noncommutative variety defined by $A$~\cite{Artin1994NoncommutativeSchemes}. Moreover, since $A$ is coherent,   
$\mathsf{gr}(A)$ is an abelian Serre subcategory, so that one can define the 
category of {\em coherent } sheaves as $\mathsf{qgr}(A):= \mathsf{gr}(A)/ \mathsf{tors} (A)$ \cite{Polishchuk2005NoncommutativeAlgebras}. 
In this paper, we consider 
the bounded derived category
$\mathbf{D}^b (\mathsf{qgr}(A))$ and the Serre twist functor $S$ on it induced by the grading shift $s$ on $\mathsf{gr}(A)$, that is, the functor $s: (sM)_i = M_{i+1}$ for a graded $A$-module $M$ and an integer $i$.

In noncommutative projective geometry, one of the most important classes of algebras is the class of regular algebras. The first version of the definition is due to Artin and Schelter~\cite{artin1987graded}.
Here we cite a more general definition from~\cite{minamoto2011structure}.

\begin{definition}
Consider a graded algebra $A = A_0 \oplus A_1 \oplus \dots$ and denote by $R = A_0$ its subalgebra of degree zero. The algebra $A$ is called regular if it has fnite global dimension (say, $d$) and 
 satisfies the
following Gorenstein property:
$$
\ext^i_A(R_A,R_A) \cong\left\{ \begin{array}{ll} 0, & i\ne d, \\
\Hom_k (R,k) [l] \mbox{ for some } l\in \inte ,& i =d .
\end{array}  \right.
$$
A regular algebra of polynomial growth is called Artin-Schelter (AS) regular.
\end{definition}

A well-known conjecture~\cite{artin1987graded} claims that all connected AS-regular
algebras are Noetherian. This conjecture holds in all known cases~\cite[Sec.~3.3]{rogalski2024artin}.

For algebras of exponential growth, A.~Bondal conjectured that all connected regular algebras are (graded) coherent.
This conjecture holds in the case $d=2$~\cite{piontkovski2008coherent} and in a number of examples with $d=3$. However, it fails for $d\ge 4$~\cite[Th.6.2.2]{gelinas2018contributions}.

\subsection{Entropy in triangulated categories}

\label{subs:category_ent_def}

The notion of the {\em entropy} for (exact) endofunctors of triangulated categories (having a split generator)
has been introduced by Dimitrov, Haiden, Katzarkov, and Kontsevich~\cite{Dimitrov2014DynamicalCategories}. In a sense, it measures the exponential growth rate of an endofunctor. 
In the case of polynomial growth, it is measured by the {\em polynomial entropy} of the endofunctor introduced later by Fan, Fu, and Ouchi in \cite{fan2021categorical}.
For some endomorphisms of projective varieties, the categorical entropy of the induced endofunctors of the derived category of coherent sheaves is connected with topological entropy, see
\cite{Dimitrov2014DynamicalCategories,Kikuta2017OnCurves,Kikuta2019OnEntropy,Yoshioka2020CategoricalSurfaces,fan2021categorical}
and others. 

For triangulated categories and thick subcategories we refer to  
\cite{Weibel2013AnAlgebra} and
\cite{Neeman2014TriangulatedCategories.AM-148}.

Let $\mathsf{C}$ be a triangulated category.  Recall that a full triangulated subcategory $\mathsf D$ is \textsf{thick} if it is closed under taking direct summands and isomorphisms. 
An object $\mathcal{G}$ is called a (classical, or split) generator of  $\mathsf C$ if the
smallest thick  triangulated subcategory of   $\mathsf C$ containing 
$\mathcal{G}$ 
is $\mathsf C$ itself.  

\begin{definition}(\cite{Dimitrov2014DynamicalCategories}, Definition 2.1)
\label{def:complexity}
Let $\mathsf{C}$ be a triangulated category with
a generator $\mathcal G$. 
Then for each object $E$ of $\mathsf{C}$  there is an object $E'$ and a tower of distinguished triangles
\begin{eqnarray}
\label{eq: tower}
\xymatrix@=.4cm{
  E_0 \ar[rr]^{} & &  E_1 \ar[rr]^{} \ar[ld]^{} & &E_2 \ar[r]^{} \ar[ld]^{}&\cdots \ar[r]^{}& E_{p-1}\ar[rr]^{} & &E_p\cong E\oplus E'  \ar[ld]^{}  \\
  &  \mathcal{G}[n_1] \ar@{-->}[lu]_{} & & \mathcal{G}[n_2]  \ar@{-->}[lu]_{}& & \cdots  & & \mathcal{G}[n_p] \ar@{-->}[lu]_{}
   }
    \end{eqnarray} 
with 
$E_0 = 0$, $p\geq 0$, and $n_i\in \mathbb{Z}$.
Let $t$ be a real number. To each tower of distinguished triangles of the form~(\ref{eq: tower})  
we associate the exponential
sum $\sum_{i=1}^p e^{n_it}$.
Let $S_t\subset \mathbb{R}$ be the set of all such sums for a given $t$. The \textsf{complexity} of $E$ with respect
to $\mathcal{G}$ is the function $\delta_t(\mathcal{G}, E):\mathbb{R}\rightarrow [0,+\infty]$
of $t$, given by $\delta_t(\mathcal{G}, E)=  \inf 
 \ S_t$.
\end{definition}

The following two functions measure the growth rate of the complexity of an endofunctor $F : \mathsf C \rightarrow \mathsf C$. 

\begin{definition}\cite[Definition~2.4]{Dimitrov2014DynamicalCategories}, \cite[Section 1.2]{fan2021categorical}
\label{def:cat_entropy}
Let $F : \mathsf C \rightarrow \mathsf C$ be an exact endofunctor of a triangulated category $\mathsf C$ with generator $\mathcal G$. 

The \textsf{entropy} and the \textsf{polynomial entropy} of $F$ are the functions $\mathsf h_t(F), \hpol_t(F) : \mathbb{R} \rightarrow [-\infty,+\infty)$ of $t$ given by
\begin{equation*}
\mathsf{h}_t( F)= \lim_{n \to \infty} \frac{1}{n} \log \delta_t(\mathcal{G}, F^n\mathcal{G}),
\end{equation*}
\begin{equation*}
\hpol_t(F) = \limsup_{n \to \infty} \frac{\log \delta_t (\mathcal{G}, F^n (\mathcal{G})) - n \mathsf{h}_t (F)}{\log n}. 
\end{equation*}
\end{definition}

The equality $\mathsf{h}_t(F)=0$ implies the subexponential growth of the complexity $ \delta_t (\mathcal{G}, F^n (\mathcal{G}))$. In this case we get simply 
\begin{equation}
\label{eq: pol_ent_pol_growth}
\hpol_t(F) = \limsup_{n \to \infty} \frac{\log \delta_t (\mathcal{G}, F^n (\mathcal{G}))}{\log n}. 
\end{equation}
In particular, the value at zero $\hpol(F):= \hpol_0(F)$ measures in this case the rate of polynomial growth of the integer-valued sequence $\delta(n)=\delta_0 (\mathcal{G}, F^n (\mathcal{G}))$.

It is shown in
~\cite[Lemma~2.5]{Dimitrov2014DynamicalCategories} and~\cite[Lemma~2.6]{fan2021categorical}
that $\mathsf{h}_t(F)$ and $\hpol_t(F) $ are well-defined, i.e., the limits exist and are independent of the choice of the generator $\mathcal G$.


\begin{definition}[\cite{fan2021categorical}]
Let $M$ and $N$ be two objects of  a triangulated category $\mathsf C$. 
The $\ext$-distance function from $M$ to $N$ is a function from $\real$ to $\real_{\ge 0} \cup \{\infty\}$ 
$$
\epsilon_t(M,N) := \delta_t(k, \RHom(M,N)) = \sum_{m\in \inte} \dim_k \ext^m (M, N) e^{-mt}.
$$ 
\end{definition}

In the smooth case, the above distance is essentially equivalent to the complexity.

\begin{theorem}[\cite{Dimitrov2014DynamicalCategories}, Theorem~2.5]
\label{th: smooth_ext_distance}
For each triangulated category $\mathsf C$ there exists a function $C_1: \real \to \real_{\ge 0}$ 
such that for all $M,N \in \Ob \mathsf C$,
$$
 \delta_t(M,N) \le C_1(t) \epsilon_t(M,N).
$$
Moreover, suppose in addition that the triangulated category $\mathsf C$ admits an $A_\infty$ enhancement which is {\em saturated}, that is,  
Morita equivalent to a smooth and compact $A_\infty$-algebra.
Then  there exists another function $C_2: \real \to \real_{\ge 0}$ 
such that for all $M,N \in \Ob \mathsf C$,
$$
C_2(t) \epsilon_t(M,N) \le
 \delta_t(M,N).
$$
\end{theorem}

Now, consider the (positive) orbit algebra of an endofuntor $F$ on the category $\mathsf C$, that is, the bigraded algebra 
$$Or = Or(F,\mathcal{O} ) = \bigoplus_{n\ge 0} \bigoplus_{m\in \inte} Or_{n,m} \mbox{ with }Or_{n,m}  = \ext^m (\mathcal{O},F^n\mathcal{O}),
$$
where $\mathcal{O}$ is an object of  $\mathsf C$. Then the $n$-th graded component $Or_{n} = \bigoplus_m Or_{n,m} $ is $\RHom^\cdot (\mathcal{O},F^n\mathcal{O}) $,  so that its Hilbert series of the variable $e^{-t}$ is 
$$
H_{Or_{n} } (e^{-t}) = \epsilon_t(\mathcal{O},F^n\mathcal{O}) .
$$
Then Theorem~\ref{th: smooth_ext_distance} implies the following connections between the algebraic entropies and categorical entropies. It is essentially proved in \cite[Theorem~2.5]{Dimitrov2014DynamicalCategories} and \cite[Lemma 2.7]{fan2021categorical}.

\begin{theorem}
\label{th: smooth_quoted}
(a) Let $\mathsf C$ be a triangulated category with generator $\mathcal{O}$, let $F$ be an exact endofunctor on $\mathsf C$. Denote by $Or = Or(F, \mathcal{O})$  the corresponding orbit algebra.
Then we have the inequality
$$
\mathsf{h}_t(F)  \le \mathsf{h}_t(Or). 
$$
Moreover, if $\mathsf{h}_t(F)  =   \mathsf{h}_t(Or) =0$, then  
$$
\hpol_t(F)  \le  \hpol_t(Or).
$$

(b) Suppose in addition that the triangulated category $\mathsf C$ admits an $A_\infty$ enhancement which is {\em saturated}, that is,  
Morita equivalent to a smooth and compact $A_\infty$-algebra. Then the equalities
$$
\mathsf{h}_t(F)  = \mathsf{h}_t(Or)
\mbox{ and }
\hpol_t(F)=  \hpol_t(Or)
$$
hold. 
\end{theorem}

\section{Smooth commutative and noncommutative varieties}

\label{sec: smooth}

If the graded algebra $A$ is smooth, the categorical and polynomial entropies of the Serre twist can be calculated by Theorem~\ref{th: smooth_quoted}. The benchmark case here is the coordinate ring of a smooth projective 
variety~\cite{Dimitrov2014DynamicalCategories, fan2021categorical}.

\begin{proposition}
\label{prop: smooth_proj}
Let $X$ be a smooth projective variety with the coordinate ring $A$ (so that $\dim X = \gkd A - 1$). Then 
$
\mathsf{h}_t (S) = 0 =  \mathsf{h} (A) 
$
and
$$
\hpol_t(S) = \gkd A-1. 
$$
\end{proposition}

\begin{proof}
By~\cite[Lemma~6.5]{fan2021categorical}, for a nef and big line bundle $L$ on $X$ one have 
$\hpol_t(-\otimes L) = \dim X$. Since the line bundle ${\mathcal O}_X(1)$ is nef and big, we have the polynomial entropy value
$\hpol_t(S) =  \hpol_t(- \otimes {\mathcal O}_X(1)) = \gkd A-1$. The categorical entropy of this functor is zero by~\cite[Lemma~2.13]{Dimitrov2014DynamicalCategories}.
\end{proof}

Now, consider noncommutative regular algebras.
We use here the following theorem~\cite[Theorem 3.7]{elagin2020smoothness}:
if, for a finite-dimensional $k$-algebra $R$, the semisimple algebra $R/\rad R$ is separable over $k$,
then the category $\mathbf{D}^b(\mathsf{mod-}R)$ admits a smooth and compact DG-enhancement. 
The separability condition here is very weak; in particular, it always holds if the field $k$ is perfect. 

\begin{proposition}
\label{prop:AS-reg_smooth}
Suppose that $A$ is a 
 graded coherent regular algebra.  
Suppose in addition that  the  algebra $A_0/\rad A_0$ is separable over $k$
(this condition always holds if either $k$ is a perfect field or $A$ is connected).
Then the triangulated category 
$\mathbf{D}^b(\mathsf{qgr}A)$
admits a smooth and compact DG-enhancement.
\end{proposition}

\begin{proof}
By \cite[Theorem 4.14]{minamoto2011structure}, the triangulated category $\mathbf{D}^b(\mathsf{qgr}A)$ is equivalent 
to the category $\mathbf{D}^b(\mathsf{mod-}\nabla A)$ for certain finite-dimensional algebra $\nabla A$.
By \cite[Theorem 3.7]{elagin2020smoothness}, this category admits a smooth DG-enhancement by bounded above complexes of finitely generated projective $\nabla A$-modules with bounded cohomology (the condition that $\nabla A/\rad \nabla A$ is a separable extension of $k$ holds since $\nabla A/\rad \nabla A = (A_0/\rad A_0)^l$ with $l>0$, by \cite[Definition~4.7]{minamoto2011structure})). Since $A$ has finite global dimension, the category
$\mathbf{D}^b(\mathsf{mod-}\nabla A)$  has finite global dimension as well. It follows that  
the above DG-enhancement is compact. 
\end{proof}

\begin{corollary}
\label{cor:reg}
Suppose that a 
graded 
algebra $A$ is regular and coherent. Suppose in addition that either $k$ is a perfect field or $A$ is connected 
(or, more generally, that the $k$-algebra $A_0/\rad A_0$ is separable).  
Then the category 
$\mathbf{D}^b(\mathsf{qgr}A)$ admits a split generator. Moreover, we have 
$$
\mathsf{h}_t (S) =  \mathsf{h} (A). 
$$
If $A$ has polynomial growth (that is, $A$ is Artin--Schelter regular), then
$$
\hpol_t(S) = \gkd A-1. 
$$
\end{corollary}

\begin{proof}
By~\cite[Proposition~4.3]{minamoto2011structure}, the triangulated category
$\mathbf{D}^b(\mathsf{qgr}A)$ admits a classical generator of the form $G =\bigoplus_{i=0}^{l-1} {\mathcal O}(i)$
for some $l>0$ (we denote by $\mathcal O$ the image of the module $A_A$ in $\mathbf{D}^b(\mathsf{qgr}A)$).
In view of Proposition~\ref{prop:AS-reg_smooth} and Theorem~\ref{th: smooth_quoted}, one can calculate
the complexity of functors on $\mathbf{D}^b(\mathsf{qgr}A)$ by the ext-distance.

Denote $d = \gd A$
 and $a_n = \dim A_n$.
 According 
 to~\cite[Proposition~4.4]{minamoto2011structure}, for all $n\ge 0$
one has $\ext_{\mathsf{qgr}A}^{m} ({\mathcal O}, S^n{\mathcal O})  =0 $ for $m\ne 0$ and 
$\ext_{\mathsf{qgr}A}^{0} ({\mathcal O}, S^n{\mathcal O})  = A_n $.
Then 
we have
$$
\epsilon_t({\mathcal O}, S^n{\mathcal O})  = \sum_{m\in \inte} \dim_k \ext_{\mathsf{qgr}A}^{m} ({\mathcal O}, S^n{\mathcal O}) e^{-mt} = a_n,
$$
so that the algebra $Or(S, {\mathcal O})$ is concentrated in its part $Or(S, {\mathcal O})_{*, 0} $ which is isomorphic to $A$.
Since $G=\bigoplus_{i=0}^{l-1} {\mathcal O}(i)$ is a compact object, we have 
$$
\ext^\cdot (G, S^nG) =  \ext^\cdot(\bigoplus_{i=0}^{l-1} {\mathcal O}(i),\bigoplus_{j=0}^{l-1} S^n{\mathcal O}(j) ) 
= \bigoplus_{i=0}^{l-1} \bigoplus_{j=0}^{l-1}  \ext^\cdot( {\mathcal O}(i), S^n{\mathcal O}(j) ) 
$$
$$
=  \bigoplus_{i=0}^{l-1} \bigoplus_{j=0}^{l-1} \ext^\cdot( {\mathcal O}, S^n{\mathcal O}(j-i) ), 
$$
so that the orbit algebra $Or(S,G) = \bigoplus_{n\ge 0} \ext^\cdot (G, S^nG)$ for the generator $G$
is isomorphic to the finite free module 
$\bigoplus_{i=0}^{l-1} \bigoplus_{j=0}^{l-1}Or(S, {\mathcal O})(j-i) $
over the algebra $Or(S, {\mathcal O})\simeq A$. Then 
$$\mathsf{h}_t (S) = \mathsf{h}_t \, Or(S,G) = \mathsf{h}_t  \, Or(S, {\mathcal O}) = \mathsf{h} (A)  .
$$ 
If $A$ has subexponential growth, then 
$\mathsf{h}_t (S)  = \mathsf{h}_t (Or(S,G)) =  \mathsf{h} (A) =0$, so that
$$
\hpol_t(S) = \hpol_t \, Or(S,G) = \hpol_t \,Or(S, {\mathcal O}) = \hpol (A) =  \gkd A-1. 
 $$
\end{proof}

For example, let $Q = (V,E)$ be a finite acyclic quiver of infinite representation type. Recall that its preprojective algebra  $\Pi_kQ $ is defined as the quotient algebra $k\overline Q/(r)$,
where $\overline Q$ is the quiver obtained from $Q$ by adding an opposite edge $a^*$ for each edge $a$ of $Q$ and $r$ is the quadratic element $\sum_{a\in E} (aa^*-a^*a)$. 
The preprojective algebra $\Pi_kQ $
is regular and graded coherent with respect to the grading defined by $\deg a=0, \deg a^* =1$ for all $a\in E$~(see  \cite[Th. 4.14]{minamoto2011structure} and 
\cite[Cor. 5.4]{minamoto2012ampleness}; for another proof, see \cite[Th. 5.1]{krause2026serre}).  

\begin{corollary}
\label{cor:prepr}
Let $A = \Pi_kQ $ be the preprojective algebra of a finite acyclic quiver $Q$ of infinite representation type. Then the category $\mathbf{D}^b(\mathsf{qgr}A)$ admits a split generator and
$$
\mathsf{h}_t (S) =  \mathsf{h} (A). 
$$
If $A$ has subexponential growth, then $
\hpol_t(S) = \gkd A-1. 
$
\end{corollary}

\begin{proof}
For the preprojective algebra $A = \Pi_kQ$, its degree zero part $A_0$ is isomorphic to $kQ$, so that  $A_0 / \rad A_0 = k^m$, where $m =|V|$ is the number of vertices of $Q$. The latter algebra is a separable extension of $k$. So, one can apply  
Corollary~\ref{cor:reg}. 
\end{proof}

As an example of an explicit entropy calculations, consider standard regular algebras of dimension two. 


\begin{example}[Noncommutative projective lines]

\label{example:NCline}
A standard algebra is regular of global dimension two  if it has the form $A = k \langle x_1, \dots, x_g| f \rangle$,
where $g\ge 2$ and $f$ is a quadratic homogeneous polynomial of $x_1, \dots, x_g$ of tensor rank $g$~\cite{zhang1998non}. All such algebras are graded coherent~\cite{piontkovski2008coherent}. The corresponding noncommutative variety generalizes properties of the projective line~\cite{piontkovski2008coherent}.

Let's calculate the algebraic entropy of this algebra. By Corollary~\ref{cor:reg}, it is equal to the categorical entropy of the Serre twist.


By~\cite[Proposition 1.1]{zhang1998non}, the Hilbert series of the algebra $A$ is equal to $H_A(z) = (1-gz+z^2)^{-1}$, so that
$$
H_A(z) = \frac{1}{(1-zh_1)(1-zh_2)},
$$
where $h_1 = \frac{2}{g- \sqrt{g^2-4}} = \frac{g+ \sqrt{g^2-4}}{2}$
and  $h_2 = \frac{2}{g+ \sqrt{g^2-4}}$  are the inverses of the roots $z_1, z_2$ of the polynomial $1-gz+z^2$. 
If $g=2$, one gets $h_1 = h_2 = 1$ and
$
\dim A_n  = n+1.
$
If $g\ge 3$, we have $h_1\ge 1$ and $0< h_2 <1$. Then
$$
H_A(z)= \sqrt{g^2-4}\left(\frac{h_1}{1-zh_1} - \frac{h_2}{1-zh_2}  \right), 
$$
so that 
$$
\dim A_n = 
 \sqrt{g^2-4}\left( h_1^{n+1} - h_2^{n+1}\right)
=
 \sqrt{g^2-4} \, h_1^{n+1} +o(1).
$$
In both cases, we obtain 
$$
\mathsf{h}_t(S) =   \mathsf{h} (A) = \lim_{n\to \infty}\frac{\log \dim A_n}{n} = \log h_1  = \log \frac{g+ \sqrt{g^2-4}}{2}.
$$

If $g=2$, the algebra has zero entropy and polynomial growth. 
In this case, we have
$$
\hpol_t(S) = \gkd A-1 =  \limsup_{n\to \infty}\frac{\log \dim A_n}{\log n} = \limsup_{n\to \infty}\frac{\log (n+1)}{\log n} =1.
$$
\end{example}

\section{Inequalities for the entropies of Serre twist}

\label{sec: inec}

Suppose $A$ is a graded algebra 
over the simple algebra $\kkk = k^w$ for some positive integer $w$.  
 Suppose also that all vector spaces $\tor_i^A(\kkk,\kkk)$ for $i=0, \dots, d$ are finite-dimensional. This last condition is a very weak version of Koszulity. It holds, in particular, finitely generated right or left graded coherent algebras, for Koszul algebras, for $n$-Koszul algebras, for algebras of finite Backelin rate, etc. 

The vector spaces $\tor_i^A(\kkk,\kkk)$ are bigraded; the second $\inte$--grading is induced by the one on $A$. 
We introduce a notation for  the graded Betti numbers.
For a graded finitely generated right $A$-module $M=M_A$, we denote its graded 
and ordinary Betti numbers by $b_{ij}^M = \dim_k  \tor_i^A(M,\kkk)_j$
and $b_{i}^M = \dim_k  \tor_i^A(M,\kkk) =  \sum_j b_{ij}^M$.  
In particular, we have the notation $b_{ij}^\kkk = \dim_k  \tor_i^A(\kkk,\kkk)_j$ and 
 $b_{i}^\kkk = \dim_k  \tor_i^A(\kkk,\kkk) = \sum_j b_{ij}^A$ for the Betti numbers of the graded algebra $A$. 


\begin{lemma}
\label{lem:betti_eval}
Let $A$ be a connected  $\kkk$-algebra. 
Let $A_{\ge n} = A_n \oplus A_{n+1} \oplus \dots $ and  let $M$ denote the graded right $A$-module $A_{\ge n} [n] $. 
Let us denote $D_i = \max\{j | b_{tj}^\kkk \ne 0 \mbox{ for some } t \le i \}\in [0, \infty]$. 
Then for each $i\ge 0$ one has
$$
   b_i^M \le b_i^\kkk  (\dim A_n + \dots + \dim A_{n+D_i}).
$$
\end{lemma}

\begin{proof}
Obviously, $M$ is generated in degree at least 0, 
so that  $b_{ij}^M = 0$ if $i \ge 0$ and  $j<0$. By the same argument with the exact sequences of $\tor$ as in the case of connected $k$-algebras~\cite[Lemma 3.5]{piontkovski2006sets}, we have $b_{ij}^M =0$ for all $j\ge D_i$.  So, the minimal $\kkk$-module $X^M = M/ M A_{\ge 1} $ of generators of $M$ is concentrated in degrees $0, \dots, D_0$. Its $k$-dimension is bounded as 
$$
b_0^M = \dim X^M \le \dim M_0 + \dots + \dim M_{D_0} = 
\dim A_n + \dots + \dim A_{n+D_0} 
$$
$$= b_0^\kkk (\dim A_n + \dots + \dim A_{n+D_0}).
$$

Consider the minimal graded projective resolution of the left module ${}_{A}{\kkk}$:
$$
{\bf F}^{\kkk}: \dots \to F_i \to \dots \to F_1\to F_0 .
$$
Here the left projective   $\kkk$-modules $F_i$ are minimally generated by the left $\kkk$-modules $X_i \simeq F_i/A_{\ge 1} F_i $ which are concentrated in nonnegative degrees, so that $F_i = A\otimes_{\kkk} X_i$ and $b_i^{\kkk} = \dim_k X_i$. 
Then for $i\ge 1$ the module $\tor_i^A(M,\kkk) $ is the i-th homology of the complex
$$
 M \otimes_A {\bf F}^{\kkk}: \dots \to M \otimes_{\kkk}   X_i \to \dots \to M \otimes_{\kkk} X_1\to M \otimes_{\kkk} X_0. 
$$
Since the module $\tor_i^A(M,\kkk) $ is concentrated in degrees in the interval $[0,D_i]$,
we have for $i>0$
$$
b_i^M = \dim_k \tor_i^A(M,\kkk)  \le \dim (M \otimes_{\kkk}   X_i )_{\le D_i} 
\le \dim X_i \cdot \dim M_{\le D_i} =  b_i^{\kkk} (\dim A_n + \dots + \dim A_{n+D_i}).
$$ 
\end{proof}

The following theorem is a variation of~\cite[Lemma~4.2.4]{Bondal2003GeneratorsGeometry}. Here we denote by the same letter ${\mathcal O}$ both the object of $\Qgr A$ corresponding to the right module $A_A$ and its image in the derived category $D^b(\Qgr A)$. 

\begin{theorem}
\label{th:main_finite_gd}
Suppose that a connected  $\kkk$-algebra $A$ 
is such that $b_i^{\kkk} < \infty$ for all $i\ge 0$ and 
the abelian category $\Qgr A$ has finite homological dimension $d$
(so, $\ext^i_{\Qgr A}(-,-) = 0$ for all $i>d$).  

Then for all sufficiently large $D>0$ all objects $S^n {\mathcal O}$ ($n>0$)
belong to the thick subcategory  ${\mathcal C}$  in $D^b(\Qgr A)$ generated by the 
object
${\mathcal O}_{D} = \bigoplus_{j=-D}^0 S^j {\mathcal O}$. 
Moreover, the object $S^n {\mathcal O}_D$ belongs  to  ${\mathcal C}$
and  there exists an analytical  function $C(t)>0$ such that
$$
\delta_t({\mathcal O}_{D}, S^n {\mathcal O}_D) \le C(t)  (\dim A_{n-D} + \dots +\dim A_{n+D})
$$ 
for all $n>D$. 
\end{theorem}

\begin{proof}
Consider the minimal projective resolution of the right module $N = (A/A_{\ge n}) [n] $:
$$
{\bf F}^{N}: \dots \to F_d \to \dots \to F_1 \to F_0 . 
$$
Here $F_0 = A[n]$ and the image of the last map is $M = A_{\ge n} [n]$. 
So,  the subcomplex $\dots \to F_d\to \dots \to F_1$ forms the minimal projective resolution 
of $M$. So, the terms $F_i$ for $i\ge 1$ here have the form $F_i = X_i \otimes_\kkk A$, where $X_i$ is a finite-dimensional graded $\kkk$-module of $k$-dimension $\dim X_{ij} = b_{ij}^N =  b_{i-1,j}^M$. 
This means that for each $i =1, \dots, d$
one has $F_i = \oplus_{s=1}^w \oplus_{n_s} P_s (-n_s) $,
where $P_s$ are the indecomposable projective $A$-modules, the total number of summands 
is $b_{i-1}^M$, and the numbers $n_s$ belong to the interval $[0, D_{i-1}]$ (where $D_{i-1}$ is defined in Lemma~\ref{lem:betti_eval}).

Let  $\pi $ denote  the natural functor from  $ \Mod\mbox{-}A$ to $ \Qgr A$. 
It is exact, so that the image $\pi {\bf F}^{N}$ of ${\bf F}^{N}$ is exact. 
It is proved in~\cite[Lemma~4.2.4]{Bondal2003GeneratorsGeometry} that the object $S^n \mathcal O = \pi M$ of $D^b(\Qgr A)$
 is a direct summand of the truncation   
$$
\pi {\bf F}^{N}_{[1..d]}:   0\to \pi F^d \to \dots \to \pi F^1 \to 0
$$
of $\pi {\bf F}^{N}$.
Then we get the tower 
\begin{eqnarray*}
\xymatrix@=.4cm{
  0 \ar[rr]^{} & &  E_1 \ar[rr]^{} \ar[ld]^{} & &E_2 \ar[r]^{} \ar[ld]^{}&\cdots \ar[r]^{}& E_{d-1}\ar[rr]^{} & &E_d\cong \pi {\bf F}^{N}_{[1..d]} \oplus X
 \ar[ld]^{}  \\
  &  \pi F_d [d-1] \ar@{-->}[lu]_{} & & \pi F_{d-1} [d-2]   \ar@{-->}[lu]_{}& & \cdots  & & \pi F_1  \ar@{-->}[lu]_{}
   }
    \end{eqnarray*} 
 for some objects  $E_1, \dots, E_{d-1}, X$ of $D^b(\Qgr A)$.
Since $\pi P_s$ is a direct summand of ${\mathcal O}$  for each $s$, we see that for each $i =1, \dots, d$
the object $\pi F_i = \oplus_{s=1}^w \oplus_{n_s} \pi P_s (-n_s) $ is a direct summand of 
${\mathcal O}_{D}^{b_i^M}$ for each $D \ge \max\{D_0, \dots, D_{d-1}\}$.  Thus, there exist 
 objects $E_1',\dots, E_d', X'$ of $D^b(\Qgr A)$ and a tower 
\begin{eqnarray*}
\xymatrix@=.4cm{
  0 \ar[rr]^{} & &  E_1' \ar[rr]^{} \ar[ld]^{} & &E_2' \ar[r]^{} \ar[ld]^{}&\cdots \ar[r]^{}& E_{d-1}'\ar[rr]^{} & &E_d'\cong S^n \mathcal O \oplus X'
\ar[ld]^{}  \\
  &  {\mathcal O}_{D}^{b_{d-1}^M} [d-1] \ar@{-->}[lu]_{} & &   {\mathcal O}_{D}^{b_{d-2}^M} [d-2] \ar@{-->}[lu]_{}& & \cdots  & &  {\mathcal O}_{D}^{b_{0}^M} \ar@{-->}[lu]_{}
   }
.
    \end{eqnarray*} 
By the definition of complexity and Lemma~\ref{lem:betti_eval}, we conclude that
$$
\delta_t({\mathcal O}_{D}, S^n {\mathcal O}) 
\le \sum_{i=1}^d b_{i-1}^M e^{it}
\le \sum_{i=1}^d b_{i-1}^\kkk e^{it}  (\dim A_n + \dots + \dim A_{n+D_i})
\le  C(t) (\dim A_n + \dots + \dim A_{n+D}),
$$
where $  C(t) =  \sum_{i=1}^d b_{i-1}^\kkk e^{it} $.
Since $S^n {\mathcal O}_{D}= \bigoplus_{j=-D}^0 S^{n+j} {\mathcal O}$, 
we obtain for all $n>D$ the desired inequalities
$$
\delta_t({\mathcal O}_{D}, S^n {\mathcal O}_{D}) = 
\sum_{j=-D}^0  \delta_t({\mathcal O}_{D}, S^{n+j} {\mathcal O})
\le C(t) (\dim A_{n-D} + \dots + \dim A_{n+D}).
$$
\end{proof}

\begin{corollary}
\label{cor:general_ineq}
In the notation of Theorem~\ref{th:main_finite_gd},
the categorical entropy of the Serre twist on ${\mathcal C}$ satisfies the inequality
$$
\mathsf{h}_t (S) \le  \entropy (A) . 
$$
If the algebra $A$ has subexponential growth, then the polynomial categorical entropy of the same functor satisfies
$$
\hpol_t (S) \le   \gkd A  -1. 
$$
\end{corollary}

\begin{proof}
Denote $a_n = \dim A_n$ for $n\ge 0$. 
By Definition~\ref{def:cat_entropy} and Theorem~\ref{th:main_finite_gd}, we get
\begin{equation*}
\mathsf{h}_t( S)= \lim_{n \to \infty} \frac{1}{n} \log \delta_t(\mathcal{O_D}, S^n\mathcal{O_D}) \le
  \limsup_{n \to \infty} 
\frac{1}{n} \log \left( C(t) (a_{n-D}+\dots +a_{n+D}) \right)
\end{equation*}
\begin{equation*}
= \limsup_{n \to \infty} \frac{1}{n} \log a_n = \entropy (A). 
\end{equation*}
It follows that for algebras of subexponential growth we have $\mathsf{h}_t( S) = 0$. Therefore,
\begin{equation*}
\hpol_t(F) = \limsup_{n \to \infty} \frac{\log \delta_t (\mathcal{O_D}, S^n (\mathcal{O_D}))}{\log n}
\le 
\limsup_{n \to \infty} \frac{\log C(t) +\log (a_{n-D}+\dots +a_{n+D}) }{\log n}
\end{equation*}
\begin{equation*}
= \limsup_{n \to \infty} \frac{\log a_{n}}{\log n} 
= \gkd (A)  -1. 
\end{equation*}
\end{proof}

Consider the subcategory ${\mathcal C}$  from Theorem~\ref{th:main_finite_gd}.
If the algebra $A$ has finite global dimension, 
it coincides with the category of compact objects in  $D^b(\Qgr A)$~\cite[Prop.~4.2.11, 2]{Bondal2003GeneratorsGeometry}. 
If, in addition,  the algebra $A$ is right graded coherent, this category is  equivalent to the category 
$D^b(\qgr A)$~\cite[Lemmata~4.3.2, 4.3.3]{Bondal2003GeneratorsGeometry}.  

\begin{corollary}
\label{cor:main_finite_dim}
Suppose that a  connected $\kkk$-algebra $A$ is right graded coherent and has finite global dimension. Then
the categorical entropy of the Serre twist on the category ${\qgr A}$ satisfies
the inequality.
$$
\mathsf{h}_t (S) \le  \entropy (A)  
$$
If the algebra $A$ has subexponential growth, then the polynomial entropy of this functor satisfies the inequality 
$$
\hpol_t (S) \le  \gkd A  -1. 
$$
\end{corollary}

\section{Monomial algebras of polynomial growth}

\label{sec: mon_pol}

Suppose $A$ is a monomial quotient of a quiver algebra $k Q$ for some finite quiver $Q$.
If $\mathcal{O, X}$ are the images of $A$ and a coherent $A$-module $X$ in $\mathbf{D}^b (\mathsf{qgr}(A))$, we have 
$$
\delta_t(\mathcal{O, X}) 
= \rk_{\bar A}(\bar X) \in {\mathbf Z}_{\ge 0},
$$
where the rank $\rk_{\bar A}(\bar X)$ 
is the minimal number of copies of the images $\bar A$ of $A$ 
in $\qgr A$ that cover the image $\bar X$ of $X$ in $\qgr A$, that is, the minimal $s$ such that there is an epimorphism ${\bar A}^{\oplus s} \to \bar X \to 0$ (equivalently, $\bar X$ is a direct summand of ${\bar A}^{\oplus s}$)~\cite[Lemma 6.4]{lu2023entropy}. 

 Let 
$S$ be the Serre twist endofunctor on $\mathbf{D}^b (\mathsf{qgr}(A))$ for the above $A$.  Then  
$$
\hpol_t(S) = \limsup_{n \to \infty} \frac{\log \delta_t (\mathcal{O}, S^n (\mathcal{O}))-n \mathsf{h}_t (S)}{\log n} = 
\limsup_{n \to \infty} \frac{\log  \rk_{\bar A}(\bar A(n)) -n \mathsf{h}_t (S)}{\log n}, 
$$
where $\bar A(n)$ is the $n$-th degree shift of $\bar A$.

Since $\mathsf{h}_t (S)$ is a constant \cite[Th. 6.5]{lu2023entropy},
we get

\begin{proposition}
For a finitely presented monomial algebra $A$, $\hpol_t(S)$ is  constant.
\end{proposition}

\begin{example}
Consider the quiver $Q$ of the form
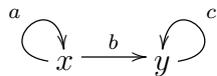
\begin{eqnarray*}
 \xymatrix{
  x \ar@(l,u)[]^a \ar@< 2pt>[r]^b  &  y  \ar@(r,u)[]_c
  }\\
    \end{eqnarray*}
Let's calculate the (polynomial) entropy of the Serre twist on the derived category of coherent sheaves on the noncommutative variety defined by its path algebra $B=kQ$.
Let $P_x$ and $P_y$ be the right ideals in $B$ generated by the vertices.  For each $m\ge 1$, the $m$-truncations of the modules $ P_x $ and $P_y$ are $(P_y)_{\ge m} \simeq P_y (-m)$ and  $(P_x)_{\ge m} \simeq B_{\ge m-1} (-1)  = (P_x\oplus P_y)_{\ge m-1} (-1)$, so that $(P_x)_{\ge m} \simeq P_x(-m) \oplus P_y^{m-1}(-m)$. Then $B_{\ge m} = (P_x\oplus P_y)_{\ge m} \simeq P_x(-m) \oplus P_y^{m}(-m) $. 
Since $B_{\ge 2m}\simeq P_x(-2m) \oplus P_y^{2m}(-2m) $, we get for the degree shift an isomorphism $B_{\ge 2m}(m) \simeq P_x(-m) \oplus P_y^{2m}(-m)$. Then
there exists an epimorphism of the truncations 
$$
B_{\ge m} \oplus B_{\ge m} \twoheadrightarrow B_{\ge 2m}(m),
$$
which induces an epimorphism $\overline B \oplus \overline B \twoheadrightarrow \overline B(m)$ in $\qgr B$. Thus, $\delta_t (\mathcal{O}, S^m (\mathcal{O})) = 
\rk_{\bar B}(\bar B(m)) \le 2$. It follows that
$$
\hpol_t(S) = \limsup_{n \to \infty} \frac{\log \delta_t (\mathcal{O}, S^n (\mathcal{O}))-\mathsf{h}_t (S)}{\log n}
=
\limsup_{n \to \infty} \frac{\log  \rk_{\bar B}(\bar B(n))}{\log n} = 0.
$$
\end{example}

\begin{theorem}
\label{th:quiver_polynomial_entropy}
Let $A$ be a finitely presented monomial algebra of at most polynomial growth. In the notation above, there exists some $C>0$ such that $\delta_t (\mathcal{O}, S^m (\mathcal{O})) \le C$ for all $m\ge 0$. As a corollary,
$$
\hpol_t(S) = 0.
$$
\end{theorem}


Before we prove the theorem, let us discuss path algebras of polynomial growth. 
It is well-known~\cite[Sect. 5.6]{Ufnarovskij1995CombinatorialAlgebra} that the growth of the  algebra $B = kQ$ is at most polynomial if and only if there are no  two simple circuits with a common vertex in $Q$. Note that the noncommutative quasi-coherent schemes defined by such algebras have been classified in~\cite{holdaway2014path}. We do not use this classification.

\begin{lemma}
\label{lem:pol_growth_quiver}
Let $b_i^m$ denotes the number of paths of length $m$ which end at the vertex $i$ of a finite quiver $Q=(V,E)$.  If the growth of the algebra $B=kQ$ is at most polynomial, then there exists  $N>0$
such that for all large enough $m$ and all $d>0$ we have 
$$
    N  b_i^{m+d} \ge b_i^{m}
$$
for each $i\in V$.
\end{lemma}

\begin{proof}
Let $s\le n!$ be the least common multiple of the lengths of all simple circuits in $Q$. 

For $m$ large enough (namely, for $m\ge |V|$), each path $p$ of length $m$ contains a simple circuit. Let $j$ be the first circuit vertex in this path. Suppose $j$ belongs to a simple circuit $c$ of length $l\le n$. Since the growth is polynomial, this circuit $c$ is uniquely defined.  If we replace the vertex $j$ of $p$ by the circuit $c$ taken $s/l$ times, we get another path 
of length $m+s$ which ends at $i$. This path is uniquely defined by $p$ and vice versa, so that we get  the inequality $b_i^{m+s} \ge  b_i^{m}$. 
It follows that for each $d>0$ with $r\simeq d \mod s $ we have $
     b_i^{m+d} \ge b_i^{m+r}$. So, it is sufficient to find 
     $N>0$ 
     such that $N b_i^{m+r} \ge  b_i^{m}$ for each $0\le r < s$.

Now, let $q$ be a path of length $m+r$ which ends at $i$. Decompose it as a composition of paths $q=q'p$, where $q'$ has length $r$. Then the path $p $ is uniquely defined by $q$, has length $m$, and ends at $i$.
Note that for each such path $p$
of length $m$ there exist no more than  $C_r$ corresponding paths $q$ of length $m+r$, where 
$C_r$ is the number of all paths of $Q$ of length $r$. Then $b_i^{m+r} \ge C_r^{-1} b_i^{m}$. Thus, for 
$
N = \max \{C_r \,|\, r =0, \dots, s-1 \}
$
we have $N b_i^{m+r} \ge b_i^{r}$. 
\end{proof}

\medskip
Now, we are ready to prove Theorem~\ref{th:quiver_polynomial_entropy}.
\medskip

\begin{proof}[Proof of Theorem~\ref{th:quiver_polynomial_entropy}]
 
 
 Let $Q_A$ be the Ufnarovski quiver of $A$. Smith and Holdaway have constructed an algebra map $f: A\to kQ_A$ which induces an equivalence of categories $\Qgr A \equiv \Qgr (kQ_A)$~\cite[Theorem~4.2]{Holdaway2012AnQuivers}. Moreover, it induces an equivalence of categories $\qgr A \equiv  \qgr (kq_A)$ and of derived categories $\qgr A \equiv  \qgr (kq_A)$~\cite[Theorem 5.5]{lu2023entropy}.  Hence, it induces an equivalence of derived categories such that the object $\mathcal{A} \in \Ob \mathbf{D}^b (\mathsf{qgr}(A))$ is mapped under this equivalence to the image of  $kQ_A$ in $\mathbf{D}^b (\mathsf{qgr}(kQ_A))$. This equivalence commutes with the Serre twist $S$, so, 
 the polynomial category entropy $h_t(S) $ of the Serre twist is the same for these two algebras. 
 Moreover, the Hilbert series of these two algebras are  equal up to a polynomial, that is, $\dim A_n = \dim (k Q_A)_n$ for $n>> 0$. It follows that the path algebra $k Q_A$ has the same polynomial growth as $A$.
So, it is sufficient to prove the theorem for the path algebra $k Q_A$ in place of $A$.

In other word, we can assume that $A= kQ$ is a path algebra. Then we can use  Lemma~\ref{lem:pol_growth_quiver}. In the notation of this lemma, 
fix $i\in V$. Since the generating function $b_i(z) = \sum_{m
 \ge 0} b_i^m z^m $ is rational and its coefficients $b_i^m $ have at most polynomial growth, then $b_i^m $ is a quasi-polynomial as a function of $m$. This means that there exist an integer $T$ and integral polynomials $f_1(m), \dots, f_T(m)$ such that  for large enough $m$, one has 
 $
 b_i^m = f_{\alpha} (m)
 $
  iff $T\, | \, (m-\alpha)$. 
 It follows from Lemma~\ref{lem:pol_growth_quiver}
 that all polynomials $f_1(m), \dots, f_T(m)$ have the same degree, say, $D$. Then there exist
two positive constants  $C_1 = C_1(i), C_2 = C_2(i)$ such that $C_1 m^D \le b_i^m  \le C_2 m^D$ for all large enough $m$. Now, for a fixed large enough  $m$, let 
 $d>0$ be large enough to fulfil the inequality $\left( (m+d)/d \right)^D < 2  $.  Then 
 $$
    b_i^{m+d}  \le C_2 (m+d)^D \le 2 C_2 d^D  \le (2 C_2/C_1) b_i^d.
 $$
 
 Let $C \ge \max\{ 2 C_2(i)/C_1(i)\,|\,  i\in V\}$ be an integer. For $ V
 = \{v_1, \dots, v_n\} $, let $P_j$ denote the elementary projective $B$-module $v_j B$ (where $j=1, \dots, n$). Then $B \simeq \oplus_{j=0}^n P_j$ as a right $B$-module. Hence $B_{\ge d} \simeq \oplus_{j=0}^n P_j^{b_j^d} (-d)$ and  
 $B_{\ge m+d} \simeq \oplus_{j=0}^n P_j^{b_j^{m+d}} (-m-d)$.  Since $b_j^{m+d}  \le C b_j^d$, there exists an epimorphism from the module $B_{\ge d}^C = \oplus_{j=0}^n P_j^{C b_j^d} (-d) $ onto  the module $
 B_{\ge m+d} (m)\simeq \oplus_{j=0}^n P_j^{b_j^{m+d}} (-d)$. It induces an epimorphism 
$\overline B^C \twoheadrightarrow \overline B(m)$ in $\qgr B$. 
 Thus, $\delta_t (\mathcal{O}, S^m (\mathcal{O})) = 
\rk_{\bar B}(\bar B(m)) \le C$. 
 \end{proof}

\bibliographystyle{abbrv}

\bibliography{mybibfile}

\end{document}